\documentclass[12pt]{article}
\usepackage{lineno,hyperref}
\modulolinenumbers[5]
\usepackage{color}

\usepackage{graphics,amsmath,amssymb}
\usepackage{amsthm}
\usepackage{amsfonts}
\usepackage{latexsym}
\usepackage{epsf}

\newcommand{\stirlings}[2]{\genfrac\{\}{0pt}{}{#1}{#2}}
\def\F{\mathcal{F}}
\def\a{\bar{A}}

\theoremstyle{plain}
\newtheorem{theorem}{Theorem}

\theoremstyle{definition}

\theoremstyle{remark}

\begin{document}

\begin{center}
\vskip 1cm{\LARGE\bf On Noncentral Tanny-Dowling Polynomials and Generalizations of Some Formulas\\
\vskip 0.1in for Geometric Polynomials}
\vskip 1cm
\large    
Mahid M. Mangontarum$^1$ and Norlailah M. Madid$^2$\\
Department of Mathematics\\
Mindanao State University--Main Campus\\
Marawi City 9700\\
Philippines \\
\href{mailto:mmangontarum@yahoo.com}{\tt $^1$mmangontarum@yahoo.com} \\
\href{mailto:mangontarum.mahid@msumain.edu.ph}{\tt $^1$mangontarum.mahid@msumain.edu.ph} \\
\href{mailto:norlailahmadid07@gmail.com}{\tt $^2$norlailahmadid07@gmail.com}
\end{center}

\vskip .2 in

\begin{abstract}
In this paper, we establish some formulas for the noncentral Tanny-Dowling polynomials including sums of products and explicit formulas which are shown to be generalizations of known identities. Other important results and consequences are also discussed and presented.
\end{abstract}

\section{Introduction}

The geometric polynomials \cite{Tanny}, denoted by $w_{n}(x)$, are defined by
\begin{equation}
w_{n}(x)=\sum_{k=0}^nk!\stirlings{n}{k}x^k,\label{def1}
\end{equation}
where $\stirlings{n}{k}$ are the well-celebrated Stirling numbers of the second kind \cite{Comtet,Stirling}. These polynomials are known to satisfy the exponential generating function
\begin{equation}
\sum_{n=0}^{\infty}w_{n}(x)\frac{z^n}{n!}=\frac{1}{1-x(e^z-1)}\label{def1.1}
\end{equation}
and the recurrence relation
\begin{equation}
w_{n+1}(x)=x\frac{d}{dx}\left[w_{n}(x)+xw_{n}(x)\right].
\end{equation}
The case when $x=1$ yields
\begin{equation}
w_n:=w_n(1)=\sum_{k=0}^nk!\stirlings{n}{k},
\end{equation}
the geometric numbers (or ordered Bell numbers). Recall that the numbers $\stirlings{n}{k}$ count the number of partitions of a set $X$ with $n$ elements into $k$ non-empty subsets. These numbers can also be interpreted as the number of ways to distribute $n$ distinct objects into $k$ identical boxes such that no box is empty. On the other hand, the numbers $k!\stirlings{n}{k}$ can be combinatorially interpreted as the number of distinct ordered partitions of $X$ with $k$ blocks, or the numbers of ways to distribute $n$ distinct objects into $k$ distinct boxes. It follows immediately that the geometric numbers count the number of distinct ordered partitions of the $n$-set $X$.

The study of geometric polynomials and numbers has a long history. Aside from the paper of Tanny \cite{Tanny}, one may also see the works of Boyadzhiev \cite{Boyad}, Dil and Kurt \cite{Dil1}, and the references therein for further readings. Benoumhani \cite{Benoumhani2} studied two equivalent generalizations of $w_n(x)$ given by
\begin{equation}
F_{m,1}(n;x)=\sum_{k=0}^nm^kk!W_{m}(n,k)x^{k}\label{def2}
\end{equation}
and
\begin{equation}
F_{m,2}(n;x)=\sum_{k=0}^nk!W_{m}(n,k)x^{k},\label{def2.1}
\end{equation}
where $W_{m}(n,k)$ denote the Whitney numbers of the second kind of Dowling lattices \cite{Benoumhani1}. These are called Tanny-Dowling polynomials and are known to satisfy the following exponential generating functions:
\begin{equation}
\sum_{n=0}^{\infty}F_{m,1}(x)\frac{z^n}{n!}=\frac{e^z}{1-x(e^{mz}-1)},\label{def4}
\end{equation}
\begin{equation}
\sum_{n=0}^{\infty}F_{m,2}(x)\frac{z^n}{n!}=\frac{e^z}{1-\frac{x}{m}(e^{mz}-1)}.\label{def5}
\end{equation}
More properties can be seen in \cite{Benoumhani1,Benoumhani2}. In a recent paper, Karg{\i}n \cite{Kargin} established a number of explicit formulas and formulas involving products of geometric polynomials, viz.
\begin{equation}
(x+1)\sum_{k=0}^n\binom{n}{k}w_k(x)w_{n-k}(x)=w_{n+1}(x)+w_n(x),\label{kargin1}
\end{equation}
\begin{equation}
\sum_{k=0}^n\binom{n}{k}w_k(x_1)w_{n-k}(x_2)=\frac{x_2w_n(x_2)-x_1w_n(x_1)}{x_2-x_1},\label{kargin2}
\end{equation}
\begin{equation}
w_n(x)=x\sum_{k=1}^n\stirlings{n}{k}(-1)^{n+k}k!(x+1)^{k-1},\label{kargin3}
\end{equation}
and
\begin{equation}
w_n(x)=\sum_{k=0}^n\stirlings{n}{k}k!x^k\frac{2^{n+1}(x+1)x^k+(-1)^{k+1}}{(2x+1)^{k+1}}.\label{kargin4}
\end{equation}
This was done with the aid of the two-variable geometric polynomials $w_k(r;x)$ defined by 
\begin{equation}
\sum_{n=0}^{\infty}w_n(r;x)\frac{z^n}{n!}=\frac{e^{rz}}{1-x(e^z-1)}.
\end{equation}

A natural generalization of $F_{m,1}(x)$ and $F_{m,2}(x)$ are the noncentral Tanny-Dolwing polynomials introduced by Mangontarum et al. \cite{Mah3} defined as
\begin{equation}
\widetilde{\F}_{m,a}(n;x)=\sum_{k=0}^{n}k!\widetilde{W}_{m,a}(n,k)x^{k},\label{def6}
\end{equation}
where $\widetilde{W}_{m,a}(n,k)$ are the noncentral Whitney numbers of the second kind with an exponential generating function given by
\begin{equation}
\sum_{n=k}^{\infty}\widetilde{\F}_{m,a}(n;x)\frac{z^n}{n!}=\frac{me^{-az}}{m-x(e^{mz}-1)}.\label{def9}
\end{equation}
Looking at \eqref{def9}, it is readily observed that
\begin{equation*}
\widetilde{\F}_{m,0}(n;x)=w_n\left(\frac{x}{m}\right),
\end{equation*}
\begin{equation*}
\widetilde{\F}_{m,-1}(n;x)=F_{m,2}(n;x)
\end{equation*}
and
\begin{equation*}
\widetilde{\F}_{1,-r}(n;x)=w_n(r,x).
\end{equation*}
The numbers $\widetilde{W}_{m,a}(n,k)$ admit a variety of combinatorial properties which can be seen in \cite{Mah3}. These numbers appear to be a common generalization of $\stirlings{n}{k}$ and $W_{m}(n,k)$, as well as other notable numbers reported by the respective authors in \cite{Belbachir,Broder,Koutras,Mangontarum1,Mangontarum2}. It is important to note that the noncentral Whitney numbers of the second kind is equivalent to the $(r,\beta)$-Stirling numbers by Corcino \cite{Corcino} and the $r$-Whitney numbers of the second kind by Mez\H{o} \cite{Mezo}.

In the present paper, we establish some formulas for the noncentral Tanny-Dowling polynomials including sums of products and explicit formulas. These formulas are shown to generalize the above-mentioned identities obtained by Karg{\i}n \cite{Kargin} for the geometric polynomials when the parameters are assigned with specific values. We also discuss some identities resulting from the said formulas.

\section{Formulas for Sum of Products}

Now, the exponential generating function in \eqref{def9} can be rewritten as
\begin{equation*}
\sum_{n=0}^{\infty}\widetilde{\F}_{m,a}(n;x)\frac{z^n}{n!}=\frac{1}{1-\frac{x}{m}(e^{mz}-1)}\cdot e^{-az}.
\end{equation*}
Hence, by applying \eqref{def1.1} and using Cauchy's product for two series, we obtain
\begin{eqnarray*}
\sum_{n=0}^{\infty}\widetilde{\F}_{m,a}(n;x)\frac{z^n}{n!}&=&\sum_{n=0}^{\infty}m^nw_n\left(\frac{x}{m}\right)\frac{z^n}{n!}\sum_{n=0}^{\infty}(-a)^n\frac{z^n}{n!}\\
&=&\sum_{n=0}^{\infty}\left(\sum_{k=0}^n\binom{n}{k}w_k\left(\frac{x}{m}\right)m^k(-a)^{n-k}\right)\frac{z^n}{n!}.
\end{eqnarray*}
Comparing the coefficients of $\frac{z^n}{n!}$ yields the result in the next theorem.
\begin{theorem}\label{t1}
The noncentral Tanny-Dowling polynomials $\widetilde{\F}_{m,a}(n;x)$ satisfy the following identity:
\begin{equation}	
\widetilde{\F}_{m,a}(n;x)=\sum_{k=0}^{n}\binom{n}{k}m^kw_k\left(\frac{x}{m}\right)(-a)^{n-k}.\label{1stresult}
\end{equation}
\end{theorem}
\begin{proof}[Alternative proof of Theorem \ref{t1}]
From \cite[Theorem 10]{Mah3}, the noncentral Whitney numbers of the second kind satisfy the following formula in terms of the Stirling numbers of the second kind:
\begin{equation*}
\widetilde{W}_{m,a}(n,k)=\sum_{j=0}^n\binom{n}{j}(-a)^{n-j}m^{j-k}\stirlings{j}{k}.
\end{equation*}
Multiplying both sides by $k!x^k$ and summing over $k$ gives the desired result.
\end{proof}
Before proceeding, we see that when $m=1$ and $a=-r$, \eqref{1stresult} becomes
\begin{equation*}
\widetilde{\F}_{1,-r}(n;x)=\sum_{k=0}^{n}\binom{n}{k}w_k(x)r^{n-k}:=w_n(r;x),
\end{equation*}
which is precisely an identity obtained by Karg{\i}n \cite[Equation (13)]{Kargin}. 

By applying the exponential generating function in \eqref{def9},
\begin{eqnarray*}
\sum_{n=0}^{\infty}\left[\widetilde{\F}_{m,a-m}(n;x)-\widetilde{\F}_{m,a}(n;x)\right]\frac{z^{n}}{n!}&=&\frac{me^{-(a-m)z}}{m-x(e^{mz}-1)}-\frac{me^{-az}}{m-x(e^{mz}-1)}\\
&=&\frac{m}{x}\left[\frac{me^{-az}}{m-x(e^{mz}-1)}-e^{-az}\right]\\
&=&\frac{m}{x}\sum_{n=0}^{\infty}\widetilde{\F}_{m,a}(n;x)\frac{z^{n}}{n!}-\sum_{n=0}^{\infty}{(-a)}^{n}\frac{z^n}{n!}\\
&=&\sum_{n=0}^{\infty}\frac{m}{x}\left(\widetilde{\F}_{m,a}(n;x)-{(-a)}^{n}\right)\frac{z^{n}}{n!}.
\end{eqnarray*}
Comparing the coefficients of $\frac{z^{n}}{n!}$ gives
\begin{equation*}
\widetilde{\F}_{m,a-m}(n;x)-\widetilde{\F}_{m,a}(n;x)=\frac{m}{x}\left[\widetilde{\F}_{m,a}(n;x)-{(-a)}^n\right].
\end{equation*}
The result in the next theorem follows by solving for $x\widetilde{\F}_{m,a-m}(n;x)$.
\begin{theorem}
The noncentral Tanny-Dowling polynomials $\widetilde{\F}_{m,a}(n;x)$ satisfy the following recurrence relation:
\begin{equation}
x\widetilde{\F}_{m,a-m}(n;x)=(m+x)\widetilde{\F}_{m,a}(n;x)-(-a)^nm.\label{5thresult}
\end{equation}
\end{theorem}
Setting $m=1$ and $a=-r$ in \eqref{5thresult} gives
\begin{equation*}
x\widetilde{\F}_{1,-r-1}(n;x)=(1+x)\widetilde{\F}_{1,-r}(n;x)-r^n
\end{equation*}
which is exactly the following identity \cite[Equation (14)]{Kargin}:
\begin{equation*}
xw_n(r+1;x)=(1+x)w_n(r;x)-r^n.
\end{equation*}
On the other hand, when $a=0$ and $a=m$ in \eqref{5thresult}, we get
\begin{equation}
x\widetilde{\F}_{m,-m}(n;x)=(m+x)w_n\left(\frac{x}{m}\right)\label{6thresult}
\end{equation}
and
\begin{equation}
(m+x)\widetilde{\F}_{m,m}(n;x)=xw_n\left(\frac{x}{m}\right)-(-m)^{n+1},\label{7thresult}
\end{equation}
respectively. Applying \eqref{1stresult} yields
\begin{equation}
xm^n\sum_{k=0}^n\binom{n}{k}w_k\left(\frac{x}{m}\right)=(m+x)w_n\left(\frac{x}{m}\right)\label{8thresult}
\end{equation}
and 
\begin{equation}
(m+x)m^n\sum_{k=0}^n\binom{n}{k}w_k\left(\frac{x}{m}\right)(-1)^{n-k}=xw_n\left(\frac{x}{m}\right)-(-m)^{n+1}.\label{9ththresult}
\end{equation}
These are generalizations of the results obtained by Dil and Kurt \cite{Dil1} using the Euler-Seidel matrix method. That is, setting $x=1$ and $m=1$ gives
\begin{equation*}
\sum_{k=0}^n\binom{n}{k}w_k=2w_n
\end{equation*}
and
\begin{equation*}
2\sum_{k=0}^n\binom{n}{k}(-1)^kw_k=(-1)^nw_n+1.
\end{equation*}

The next theorem contains a formula for the sum of product of noncentral Tanny-Dowling polynomials for different values of $a$.
\begin{theorem}
The noncentral Tanny-Dowling polynomials satisfy the following relation:
\begin{equation}
x\sum_{k=0}^n\binom{n}{k}\widetilde{\F}_{m,a_1}(k;x)\widetilde{\F}_{m,a_2}(n-k;x)=\widetilde{\F}_{m,\a}(n+1;x)+\a\widetilde{\F}_{m,\a}(n;x),\label{10ththresult}
\end{equation}
where $\a=a_1+a_2+m$ for real numbers $a_1$ and $a_2$.
\end{theorem}
\begin{proof}
We start by taking the derivative of \eqref{def9} with respect to $z$. That is,
\begin{equation*}
\frac{\partial}{\partial z}\left(\frac{me^{-az}}{m-x(e^{mz}-1)}\right)=\frac{me^{-az}}{m-x(e^{mz}-1)}\cdot\frac{xme^{mz}}{m-x(e^{mz}-1)}-\frac{ame^{-az}}{m-x(e^{mz}-1)}.
\end{equation*}
Replacing $a$ with $\a=a_{1}+a_{2}+m$ yields
\begin{equation*}
\frac{\partial}{\partial z}\left(\frac{me^{-\a z}}{m-x(e^{mz}-1)}\right)=\sum_{n=k}^{\infty}\widetilde{\F}_{m,\a}(n+1;x)\frac{z^{n}}{n!}
\end{equation*}
in the left-hand side while we get
\begin{eqnarray*}
\frac{me^{-\a z}}{m-x(e^{mz}-1)}\cdot\frac{xme^{mz}}{m-x(e^{mz}-1)}&=&\frac{me^{-a_1z}}{m-x(e^{mz}-1)}\cdot\frac{me^{-a_2z}}{m-x(e^{mz}-1)}\\
&=&x\sum_{n=k}^{\infty}\sum_{k=0}^n\binom{n}{k}\widetilde{\F}_{m,a_1}(k;x)\widetilde{\F}_{m,a_2}(n-k;x)\frac{z^n}{n!}
\end{eqnarray*}
and
\begin{equation*}
\frac{\a me^{-\a z}}{m-x(e^{mz}-1)}=\a\cdot\sum_{n=k}^{\infty}\widetilde{\F}_{m,\a}(n;x)\frac{z^n}{n!}
\end{equation*}
in the right-hand side. Combining the above equations and comparing the coefficients of $\frac{z^n}{n!}$ gives the desired result.
\end{proof}

When $a_1=a_2=0$ in \eqref{10ththresult},
\begin{equation*}
x\sum_{k=0}^n\binom{n}{k}w_k\left(\frac{x}{m}\right)w_{n-k}\left(\frac{x}{m}\right)=\widetilde{\F}_{m,m}(n+1;x)+m\widetilde{\F}_{m,m}(n;x).
\end{equation*}
Applying \eqref{5thresult} to the right-hand side of this equation gives
\begin{equation*}
x\sum_{k=0}^n\binom{n}{k}w_k\left(\frac{x}{m}\right)w_{n-k}\left(\frac{x}{m}\right)=\frac{xw_{n+1}\left(\frac{x}{m}\right)-(-m)^{n+2}}{m+x}+m\frac{xw_n\left(\frac{x}{m}\right)-(-m)^{n+1}}{m+x}
\end{equation*}
which can be simplified into the following identity:
\begin{equation}
(m+x)\sum_{k=0}^n\binom{n}{k}w_k\left(\frac{x}{m}\right)w_{n-k}\left(\frac{x}{m}\right)=w_{n+1}\left(\frac{x}{m}\right)+mw_n\left(\frac{x}{m}\right).\label{11thresult}
\end{equation}
Obviously, this identity boils down to the result obtained by Karg{\i}n \cite{Kargin} in \eqref{kargin1} when $m=1$.

\begin{theorem}
For $x_1\neq x_2$, the following formula holds:
\begin{equation}
\sum_{k=0}^{n}\binom{n}{k}\widetilde{\F}_{m,a_1}(k;x_1)\widetilde{\F}_{m,a_2}(n-k;x_2)=\frac{x_2\widetilde{\F}_{m,a_1+a_2}(n;x_2)-x_1\widetilde{\F}_{m,a_1+a_2}(n;x_1)}{x_2-x_1}. \label{12thresult}
\end{equation}
\end{theorem}
\begin{proof}
Note that we can write
\begin{equation*}
\frac{me^{-a_1z}}{m-x_1(e^{mz}-1)}\cdot\frac{me^{-a_2z}}{m-x_2(e^{mz}-1)}=\frac{1}{x_2-x_1}\left[\frac{x_2me^{-(a_1+a_2)z}}{m-x_2(e^{mz}-1)}-\frac{x_1me^{-(a_1+a_1)z}}{m-x_1(e^{mz}-1)}\right].
\end{equation*}
Following the same method used in the previous theorem leads us to the desired result.
\end{proof}
This theorem contains a formula for the sums of products of noncentral Tanny-Dowling polynomials for different values of $x$. When $a_1=a_2=0$, \eqref{12thresult} reduces to 
\begin{equation}
\sum_{k=0}^{n}\binom{n}{k}w_k\left(\frac{x_1}{m}\right)w_{n-k}\left(\frac{x_2}{m}\right)=\frac{x_2w_n\left(\frac{x_2}{m}\right)-x_1w_n\left(\frac{x_1}{m}\right)}{x_2-x_1}.
\end{equation}
It is clear to see that when $m=1$, we recover the sum of products of geometric polynomials in \eqref{kargin2}.

\section{Explicit formulas}

In Theorem \ref{t1}, we obtained an explicit formula that expresses the noncentral Tanny-Dowling polynomials in terms of the geometric polynomials. Now, with $g_n=\frac{1}{a^n}\widetilde{\F}_{m,a}(n;x)$ and $f_j=\left(\frac{m}{a}\right)^jw_j\left(\frac{x}{m}\right)$, the binomial inversion formula
\begin{equation}
f_n=\sum_{j=0}^n\binom{n}{j}g_j\Longleftrightarrow g_n=\sum_{j=0}^n(-1)^{n-j}\binom{n}{j}f_j
\end{equation}
allows us to express the geometric polynomials $w_n\left(\frac{x}{m}\right)$ in terms of the noncentral Tanny-Dowling polynomials as follows.
\begin{equation}
w_n\left(\frac{x}{m}\right)=\frac{1}{m^n}\sum_{j=0}^n\binom{n}{j}a^{n-j}\widetilde{\F}_{m,a}(j;x).
\end{equation}
In this section, we will derive more explicit formulas for both polynomials.

Using $x-m$ in place of $x$ in \eqref{def9} gives
\begin{eqnarray*}
\sum_{n=k}^{\infty}\widetilde{\F}_{m,a}(n;x-m)\frac{z^n}{n!}&=&\frac{me^{-(-a-m)(-z)}}{m+x(e^{-mz}-1)}\\
&=&\sum_{n=k}^{\infty}\widetilde{\F}_{m,-a-m}(n;-x)\frac{(-z)^n}{n!}.
\end{eqnarray*}
By comparing the coefficients of $\frac{z^n}{n!}$, we get
\begin{equation}
\widetilde{\F}_{m,a}(n;x-m)=(-1)^n\widetilde{\F}_{m,-a-m}(n;-x).\label{15thresult}
\end{equation}
Applying \eqref{5thresult} to the right-hand side gives
\begin{equation*}
\widetilde{\F}_{m,a}(n;x-m)=(-1)^n\left[\frac{(m-x)\widetilde{\F}_{m,-a}(n;-x)-a^nm}{-x}\right].
\end{equation*}
Replacing $-x$ and $-a$ with $x$ and $a$, respectively, and solving for $\widetilde{\F}_{m,a}(n;x)$ yields
\begin{equation*}
\widetilde{\F}_{m,a}(n;x)=\frac{(-1)^nx\widetilde{\F}_{m,-a}(n;-x-m)+(-a)^nm}{m+x}.
\end{equation*}
By \eqref{def6}, we get the next theorem.
\begin{theorem}
The noncentral Tanny-Dowling polynomials satisfy the following explicit formula: 
\begin{equation}
\widetilde{\F}_{m,a}(n;x)=x\sum_{k=0}^n(-1)^{n+k}k!\widetilde{W}_{m,-a}(n,k)(m+x)^{k-1}+\frac{(-a)^nm}{m+x}.\label{16thresult}
\end{equation}
\end{theorem}
Setting $a=0$ in $\widetilde{W}_{m,a}(n,k)$ allows us to express the noncentral Whitney numbers of the second kind in terms of $\stirlings{n}{k}$. More precisely, when $a=0$ in \cite[Proposition 7]{Mah3}, we can see that
$$\widetilde{W}_{m,0}(n,k)=m^{n-k}\stirlings{n}{k}.$$
Thus, \eqref{16thresult} becomes
\begin{equation}
w_n\left(\frac{x}{m}\right)=x\sum_{k=0}^n(-1)^{n+k}k!m^{n-k}\stirlings{n}{k}(m+x)^{k-1}
\end{equation}
when $a=0$. Moreover, when $m=1$, we recover the explicit formula in \eqref{kargin3}. The expression $m^{n-k}\stirlings{n}{k}$ is actually called translated Whitney numbers of the second kind and is denoted by $\stirlings{n}{k}^{(m)}$. These numbers satisfy the recurrence relation given by \cite[Theorem 8]{Belbachir}
\begin{equation*}
\stirlings{n}{k}^{(m)}=\stirlings{n-1}{k-1}^{(m)}+mk\stirlings{n-1}{k}^{(m)}
\end{equation*}
and the explicit formula \cite[Proposition 2]{Mangontarum2}
\begin{equation*}
\stirlings{n}{k}^{(m)}=\frac{1}{m^kk!}\sum_{j=0}^k(-1)^{k-j}\binom{k}{j}(mj)^n.
\end{equation*}
More properties of these numbers can be seen in \cite{Mangontarum1}. With these, we may also write
\begin{equation}
w_n\left(\frac{x}{m}\right)=x\sum_{k=0}^n(-1)^{n+k}k!\stirlings{n}{k}^{(m)}(m+x)^{k-1},
\end{equation}
an explicit formula for the geometric polynomials $w_n\left(\frac{x}{m}\right)$ in terms of the translated Whitney numbers of the second kind.

Now, it can be shown that
\begin{equation*}
\frac{y^2-1}{2y}\left(\frac{e^{-a(2z)}}{y-e^{mz}}+\frac{e^{-a(2z)}}{y+e^{mz}}\right)=\frac{e^{-a(2z)}}{1-\left(\frac{1}{y^2-1}\right)\left(e^{m(2z)-1}\right)}.
\end{equation*}
Notice that the right-hand side is
\begin{eqnarray*}
\frac{e^{-a(2z)}}{1-\left(\frac{1}{y^2-1}\right)\left(e^{m(2z)-1}\right)}&=&\frac{me^{-a(2z)}}{m-\left(\frac{m}{y^2-1}\right)\left(e^{m(2z)-1}\right)}\\
&=&\sum_{n=0}^{\infty}2^n\widetilde{\F}_{m,a}\left(n;\frac{m}{y^2-1}\right)\frac{z^n}{n!}.
\end{eqnarray*}
Also, in the left-hand side, we have
\begin{equation*}
\frac{e^{-a(2z)}}{y-e^{mz}}=\frac{1}{y-1}\sum_{n=0}^{\infty}\widetilde{\F}_{m,2a}\left(n;\frac{m}{y-1}\right)\frac{z^n}{n!}
\end{equation*}
and
\begin{equation*}
\frac{e^{-a(2z)}}{y+e^{mz}}=\frac{1}{y+1}\sum_{n=0}^{\infty}\widetilde{\F}_{m,2a}\left(n;\frac{-m}{y+1}\right)\frac{z^n}{n!}.
\end{equation*}
Combining these equations and comparing the coefficients of $\frac{z^n}{n!}$ results to
\begin{equation*}
2^{n+1}\widetilde{\F}_{m,a}\left(n;\frac{m}{y^2-1}\right)=\frac{y+1}{y}\widetilde{\F}_{m,2a}\left(n;\frac{m}{y-1}\right)+\frac{y-1}{y}\widetilde{\F}_{m,2a}\left(n;\frac{-m}{y+1}\right).
\end{equation*}
Note that if we set $x=\frac{m}{y-1}$, then $y=\frac{m+x}{x}$. Hence, skipping the tedious computations allow us to write
\begin{equation*}
(m+2x)\widetilde{\F}_{m,2a}(n;x)=2^{n+1}(m+x)\widetilde{\F}_{m,a}\left(n;\frac{x^2}{m+2x}\right)-m\widetilde{\F}_{m,2a}\left(n;\frac{-mx}{m+2x}\right).
\end{equation*}
The next theorem is obtained by applying \eqref{def6}.
\begin{theorem}
The noncentral Tanny-Dowling polynomials satisfy the following explicit formula:
\begin{equation}
\widetilde{\F}_{m,2a}(n;x)=\sum_{k=0}^nk!x^k\left[\frac{2^{n+1}(m+x)x^k\widetilde{W}_{m,a}(n,k)+(-m)^{k+1}\widetilde{W}_{m,2a}(n,k)}{(m+2x)^{k+1}}\right].
\end{equation}
\end{theorem}
Since it is already known that $\widetilde{W}_{m,0}(n,k)=\stirlings{n}{k}^{(m)}$, then when $a=0$, the right-hand side can be expressed in terms of the translated Whitney numbers of the second kind. That is,
\begin{equation}
w_n\left(\frac{x}{m}\right)=\sum_{k=0}^nk!x^k\stirlings{n}{k}^{(m)}\left[\frac{2^{n+1}(m+x)x^k+(-m)^{k+1}}{(m+2x)^{k+1}}\right].
\end{equation}
Lastly, when $m=1$, we recover the explicit formula in \eqref{kargin4}.

Finally, we will end by mentioning an explicit formula for $\widetilde{\F}_{m,a}(n;x)$ established in \cite[Theorem 19]{Mah3} that is given by
\begin{equation}
\widetilde{\F}_{m,a}(n;x)=\frac{m}{m+x}\sum_{k=0}^{\infty}\left(\frac{x}{m+x}\right)^k(mk-a)^n.
\end{equation}
This explicit formula entails interesting particular cases. For instance, when $a=0$, 
\begin{equation}
w_n\left(\frac{x}{m}\right)=\frac{m^{n+1}}{m+x}\sum_{k=0}^{\infty}\left(\frac{x}{m+x}\right)^kk^n.
\end{equation}
When $m=1$ and then $x=1$, we get formulas for the ordinary geometric polynomials and numbers. That is,
\begin{equation}
w_n(x)=\frac{1}{x+1}\sum_{k=0}^{\infty}\left(\frac{x}{x+1}\right)^kk^n
\end{equation}
and
\begin{equation}
w_n=\sum_{k=0}^{\infty}\frac{k^n}{2^{k+1}}.
\end{equation}

\end{document}